\documentclass[11pt,leqno]{amsart}
\usepackage{amssymb}
\usepackage{graphicx, color, multicol}

\begin{document}
\theoremstyle{plain}
\newtheorem{thm}{Theorem}
\newtheorem{prop}{Proposition}[section]
\newtheorem{lem}[prop]{Lemma}
\newtheorem{clry}[prop]{Corollary}
\newtheorem{deft}[prop]{Definition}
\newtheorem{hyp}{Assumption}

\theoremstyle{definition}
\newtheorem{rem}[prop]{Remark}
\numberwithin{equation}{section}
\newcommand{\R}{\mathbb{R}}
\newcommand{\Z}{\mathbb{Z}}
\newcommand{\N}{\mathbb{N}}
\newcommand{\C}{\mathbb{C}}
\def\virgp{\raise 2pt\hbox{,}}

\let\et=\eta
\let\al=\alpha
\let\b=\beta
\let\g=\gamma
\let\de=\delta
\let\eps=\varepsilon
\let\k=\kappa
\let\z=\zeta
\let\lam=\lambda
\let\r=\rho
\let\s=\sigma
\let\f=\phi
\let\vf=\varphi
\let\p=\psi
\let\om=\omega
\let\G= \Gamma
\let\D=\Delta
\let\Lam=\Lambda
\let\S=\Sigma
\let\Om=\Omega
\let\wt=\widetilde
\let\wh=\widehat
\let\convf=\leftharpoonup
\let\tri\triangle

\def\cA{{\mathcal A}}
\def\cB{{\mathcal B}}
\def\cC{{\mathcal C}}
\def\cD{{\mathcal D}}
\def\cE{{\mathcal E}}
\def\cF{{\mathcal F}}
\def\cG{{\mathcal G}}
\def\cH{{\mathcal H}}
\def\cI{{\mathcal I}}
\def\cJ{{\mathcal J}}
\def\cK{{\mathcal K}}
\def\cL{{\mathcal L}}
\def\cM{{\mathcal M}}
\def\cN{{\mathcal N}}
\def\cO{{\mathcal O}}
\def\cP{{\mathcal P}}
\def\cQ{{\mathcal Q}}
\def\cR{{\mathcal R}}
\def\cS{{\mathcal S}}
\def\cT{{\mathcal T}}
\def\cU{{\mathcal U}}
\def\cV{{\mathcal V}}
\def\cW{{\mathcal W}}
\def\cX{{\mathcal X}}
\def\cY{{\mathcal Y}}
\def\cZ{{\mathcal Z}}

\renewcommand{\phi}{\varphi}
\renewcommand{\d}{\partial}
\newcommand{\dd}{\mathrm{d}}
\newcommand{\e}{\mathrm{e}}
\newcommand{\re}{\mathop{\rm Re} }
\newcommand{\im}{\mathop{\rm Im}}
\newcommand{\CO}{\mathcal{C}^{\infty}_0}
\renewcommand{\O}{\mathcal{O}}
\newcommand{\supp}{\mathop{\rm supp}}
\newcommand{\sech}{\mathop{\rm sech}}
\newcommand{\tr}{\mathop{\rm Tr}}
\newcommand{\grad}{\mathop{\rm grad}\nolimits}
\newcommand{\vol}{\mathop{\rm vol}\nolimits}
\newcommand{\res}{\mathop{\rm Res}\nolimits}
\newcommand{\id}{\mathop{\rm Id}}
\newcommand{\DN}{\mathop{\mathcal{N}}\nolimits}
%
%
%

\title[]{Stability estimates for an inverse scattering problem at high frequencies}
\date{\today}

\author[Ammari]{Habib Ammari}
\address{Department of Mathematics and Applications,
{\'E}cole Normale Sup{\'e}rieure, 45 Rue d'Ulm, F-75005 Paris, France}
\email{habib.ammari@ens.fr}
\thanks{This work was supported by the ERC Advanced Grant Project
MULTIMOD--267184 and the Institut Universitaire de France.}

\author[Bahouri]{Hajer Bahouri}
\address{Universit{\'e} Paris Est-Cr\'eteil, Umr 8050 Lama, 61 avenue du G\'en\'eral de Gaulle, F-94010 Cr\'eteil Cedex, France}
\email{hbahouri@math.cnrs.fr}

\author[Dos Santos Ferreira]{David Dos Santos Ferreira}
\address{Universit{\'e} Paris 13, Cnrs, Umr 7539 Laga, 99 avenue Jean-Baptiste Cl\'ement, F-93430 Villetaneuse, France}
\email{ddsf@math.univ-paris13.fr}

\author[Gallagher]{Isabelle Gallagher}
\address{Institut de Math{\'e}matiques Umr 7586, Universit{\'e} Paris VII, 175 rue du Chevaleret, F-75013 Paris, France}
\email{Isabelle.Gallagher@math.jussieu.fr}

\subjclass[2000]{35R30, 35B30}

\keywords{Inverse scattering problem, Lipschitz stability,
resolution}

\begin{abstract}
We consider an inverse scattering problem and its near-field approximation at
   high frequencies. We first prove, for both problems, Lipschitz stability results
   for determining the low-frequency component of the potential.
   Then we show that, in the case of a radial potential supported sufficiently near the
   boundary, infinite resolution can be achieved from measurements of the
   near-field operator in the monotone case.
\end{abstract}
\maketitle
\setcounter{tocdepth}{1}
\tableofcontents
%
%
\section*{Introduction}
%


The first aim of this paper is to establish Lipschitz stability
results for the inverse scattering problem of determining the
low-frequency component (lower than the operating frequency) of
the compactly supported potential from scattering or near-field
measurements. It is known that, in general, the problem is
exponentially unstable \cite{A1,A2,A, J, mandache}. However,
taking advantage of a priori information may improve stability and
give accurate reconstruction algorithms \cite{lnm1,AK, mathcomp2}.
The Lipschitz stability results proved in this paper together with
the recent analysis of the local convergence of the nonlinear
Landweber iteration in \cite{otmar} show that the low-frequency
component of the potential can be determined from the data in a
linearly stable way. Moreover, they precisely quantify the
resolution limit, which is defined as the characteristic size of
the smallest oscillations in the potential that can be stably
recovered from the data. Since Rayleigh's work, it has been
admitted that the resolution limit in inverse scattering is of
order $\pi$ over the operating frequency \cite{born}. This is
nothing else than a direct application of the uncertainty
principle in inverse scattering \cite{bertero, chen, donoho,
slepian83}. It is well-known that if the support of the potential
is a point support, then the reconstructed location of the point
potential from the scattering data has finite size of order of the
Rayleigh resolution limit \cite{bookimaging, bertero}. Having this
in mind, the results of this paper prove that the Fourier
transform of the potential can be reconstructed in a linearly
 stable way for all frequencies (dual variable to the
space one)  smaller than the operating frequency, and therefore,
justify the notion of resolution limit. More intriguingly, again
in view of \cite{otmar},  they prove that the stability of the
reconstruction of the potential increases at high operating
frequencies.
\smallskip
\noindent

The second aim of the paper is to show that infinite resolution
can be achieved from near-field measurements. Here, the near-field
operator approximates Sommerfeld's radiation condition and is
equivalent to the measurements of the Cauchy data at a finite
distance. Moreover, if the potential is supported near the
boundary, then infinite resolution can be achieved in the monotone
case. In fact, a Lipschitz stability result holds  for both the
low and high frequency components of the potential. It should be
noted that the scattering amplitude can be recovered from the
near-field operator. However, approximating the near-field
operator from the scattering amplitude is a severely ill-posed
problem \cite{karp,No1,No2,T2} and therefore cannot be of any
practical and realistic use. It was shown in \cite{karp} that in
order to compute the near-field operator from the scattering
amplitude one needs to differentiate the scattering amplitude an
infinite number of times.

\smallskip
\noindent The results of this paper extend to medium scattering
the recent results in~\cite{SIIMS, AGS1, zhao}, where a stability
and resolution analysis was performed for linearized conductivity
and wave imaging problems. They can be also used to justify the
hopping (or continuation in the frequency) reconstruction
algorithms proposed in \cite{bao1, chen, CR, Coifman}.

\smallskip
\noindent In connection with our results in this paper, we   also
refer to the works by Isakov~\cite{victor} and Isakov and
Kindermann \cite{victor2}, Bao, Lin, and Triki \cite{bao2,
bao2num}, Nagayasu, Uhlmann and Wang~\cite{gunther, gunther2}, as
well as Derveaux, Papanicolaou, and Tsogka~\cite{DPT}.
In~\cite{victor,victor2}, an evidence of increasing stability in
wave imaging when frequency is growing was given. In \cite{bao2},
stability estimates for the inverse source problem were
established and the conversion of the logarithmic type stability
to a Lispchitz one first proved. Numerical results to illustrate
the stability of the source reconstruction problem were presented
in \cite{bao2num}. In \cite{bao3},  Lipschitz stability estimates
for the time-dependent wave equation were obtained.
In~\cite{gunther}, a stability estimate for a linearized
conductivity problem was derived and its dependence on the depth
of the inclusion highlighted.  In \cite{gunther2}, it is shown
that the ill-posedness of the inverse acoustic problem decreases
when the frequency increases and the stability estimate changes
from logarithmic type for low frequencies to a Lipschitz type for
large frequencies. In~\cite{DPT}, the enhancement of resolution in
the near-field was studied and numerically illustrated. Our
results in this paper confirm these important observations in a
quite general situation and precisely quantify them.

\medskip
\noindent Our paper is organized as follows. Section \ref{sect2}
is devoted to the stability of the reconstruction of the potential
from the scattering amplitude (called also far-field pattern) in
the high frequency regime. Theorem~\ref{genthm} proves that the
low-frequency component of the potential can be determined in a
stable way from the scattering amplitude. The threshold frequency
determines the resolution limit. Section \ref{sect3} extends the
results of Section \ref{sect2} to the near-field measurements.
Theorem~\ref{genthm2} shows that the same results as those in
Section \ref{sect2} hold for reconstructing the potential from
measurements of the near-field operator. In Section
\ref{nearboundary} we show that we gain infinite resolution for
potentials supported near the boundary. If the potential is
supported near the boundary, then infinite resolution can be
achieved in the monotone case. Theorem~\ref{disk} provides a
Lipschitz stability result for both the low and high frequency
components of the potential. Finally, in Appendix
\ref{App:Special}, we provide useful results on Bessel's
functions.

\bigskip
\noindent
Finally, we mention that the letter~$C$ will be used to denote a universal constant
which may vary from line to line. We also use~$A\lesssim B$  to
denote an estimate of the form~$A\leq C B$  for some
constant~$C$.
We also use the classical notation~$\langle x \rangle = \sqrt{1+|x|^2}$.

\bigskip
\noindent
{\bf Acknowledgements. } We are very grateful to J. Sj\"ostrand for taking
the time to discuss  some properties of the Laplace transform with us.

%
%
\begin{section}{Far field pattern} \label{sect2}

\subsection{Definitions and notations}

Let $\Omega$ be a bounded domain in the Euclidean space $\R^d$ of dimension $d \ge 2$, let $q \in \CO(\R^d)$
be a  real-valued potential supported in $\Omega$. We use the classical notation $D=-i\d$ for derivatives and consider
the Helmholtz equation with potential
\begin{align}
\label{Far:HelmholtzPoten}
     D^2u-\lambda^2u+qu = 0
\end{align}
at frequency $\lambda \in \R^*_+:= \R_+ \setminus\{0\}$. Plane
waves $\e^{i \lambda x \cdot \omega}$ propagating along the
direction~$\omega $ in~$ S^{d-1}$ are solutions of the free
Helmholtz equation
\begin{align}
\label{Far:FreeHelmholtz}
     D^2u-\lambda^2u = 0.
\end{align}
Here $ S^{d-1}$ denotes the unit sphere in $\R^d$. More generally,
plane waves generate the set of solutions in  the space of
tempered distributions, $\cS'(\R^d)$, of the free Helmholtz
equation: all solutions of \eqref{Far:FreeHelmholtz} with
polynomial growth are superpositions of elementary plane
waves~$\e^{i \lambda x \cdot \omega}$ when~$\omega $ varies on the
sphere~$ S^{d-1}$. It will be useful later on to adopt Melrose's
notation in \cite{M} to designate these solutions: for~$g \in
\mathcal{C}^\infty(S^{d-1})$, we shall write
\begin{align*}
      \Phi_0(\lam) g (x) = \int_{S^{d-1}} {\rm e} ^{i \lambda x \cdot \omega} g(\omega) \, \dd\omega.
\end{align*}
Obviously $\Phi_0(\lam) g$ is a solution of \eqref{Far:FreeHelmholtz} which belongs to~$\cS'(\mathbb{R}^d)$.

\medskip
\noindent
To guarantee the uniqueness of solutions to the Helmholtz equation \eqref{Far:HelmholtzPoten}, one can impose conditions on the behavior of solutions at
infinity. More precisely, we are interested in solutions which can be decomposed
\begin{align}
\label{Far:Scattering}
     u = \e^{i \lambda x \cdot \omega} + u^{\rm scat} =  u^{\rm in} +  u^{\rm scat}
\end{align}
as the sum of an incoming planar wave and a scattered wave satisfying Sommerfeld's radiation condition
\begin{align}
\label{Far:Sommerfeld}
    \bigg| \bigg(\frac{\partial}{\partial |x|} - i \lambda\bigg)u^{\rm scat}\bigg| =
    o\left(\frac{1}{|x|^{\frac{d-1}{2}}}\right)  \quad \text{ as } |x| \rightarrow +\infty,
\end{align}
uniformly with respect to the direction $\theta = \frac{x}{|x|}$
at fixed frequency $\lambda \in \R^*_+$. The following result from
\cite[Lemma 2.4]{M} holds.

\begin{prop}
\label{Far:ExistenceThm}
     There exists a unique solution to the Helmholtz equation~\eqref{Far:HelmholtzPoten} with potential $q \in \CO(\Omega)$
     of the form
     \begin{equation}\label{decphiq}
       \phi_q(x,\omega,\lambda) = \e^{i \lambda x \cdot \omega} + \phi_q^{\rm
       scat}(x,\omega,\lambda),
     \end{equation}
     where the scattered wave $\phi_q^{\rm scat}$ satisfies Sommerfeld's radiation condition~\eqref{Far:Sommerfeld}
     and which is given by
         \begin{equation}\label{resscat} \phi_q^{\rm scat} = -R_q(\lambda)(\e^{i \lambda x \cdot \omega}q).     \end{equation}
     Here $R_q$ denotes the meromorphic continuation of the perturbed resolvent.
     Furthermore $  \phi_q$ depends smoothly on $(x,\omega,\lambda) \in \R^d \times S^{d-1} \times \R^*_+$ and is bounded.
\end{prop}
\noindent We choose to denote
      $$ \Phi_q(\lambda)g = \int_{S^{d-1}} \phi_q(x,\omega,\lambda) g(\omega) \, \dd \omega $$
the operator with kernel $\phi_q(x,\omega,\lambda)$ given by
(\ref{decphiq}).
\begin{thm}
\label{Far:ThmScat}
     The scattered wave in the solution \eqref{Far:Scattering} to the Helmholtz equation \eqref{Far:HelmholtzPoten}
     given by Proposition~{\rm\ref{Far:ExistenceThm}} assumes the form
     \begin{align}
          \phi_q^{\rm scat}(x) = \frac{\e^{i\lambda |x|}}{|x|^{\frac{d-1}{2}}} \, a_q\bigg(\frac{x}{|x|},\omega,\lambda\bigg) +
          \O\bigg(\frac{1}{|x|^{\frac{d+1}{2}}}\bigg) \quad \mbox{as } |x| \rightarrow +\infty,
     \end{align}
     where $a_q$ is a smooth function on $S^{d-1} \times S^{d-1} \times \R^*_+$ and
            $$ a_q\bigg(\frac{x}{|x|},\omega,\lambda\bigg) = -\frac{1}{2i \lambda} \bigg(\frac{\lambda}{2\pi i}\bigg)^{\frac{d-1}{2}}
                \int q(y) \phi_q(y,\omega,\lambda)  \e^{ - i \lambda \frac{x } {| x|} \cdot y}  \, \dd y. $$

    \end{thm}
\begin{proof}
     We consider the Green function $G_{\lambda}$ corresponding to the free Helmholtz equation
     \begin{align}
          (D_y^2-\lambda^2)G_{\lambda}(x,y) = \delta(x-y),
     \end{align}
subject to Sommerfeld's radiation condition
(\ref{Far:Sommerfeld}), with $\delta$ being the Dirac delta
function.

\noindent
     Let $R>0$ be large enough so that the ball of radius $R$ contains the support of $q$.
     By definition of the Green function, for all $|x| \leq R$ we have
         $$ u (x) = \int_{|y| \leq R} (D^2_y  - \lambda^2) G_\lambda(x,y) u (y) \, \dd y,$$
     and if $u$ is a solution  to the Helmholtz equation \eqref{Far:HelmholtzPoten} we deduce by Green's formula that for any $|x| \leq R$
     \begin{align*}
           u(x) &= - \int_{|y| \leq R} G_\lambda(x,y) q(y) u (y) \, \dd y \\
           &\quad - \int_{|y|=R} \left( \frac{ \partial G_\lambda}{\partial r}(x,y) \,u (y) - G_\lambda(x,y) \frac{ \partial u}{\partial r}(y)\right)
           \, \dd \sigma (y).
     \end{align*}
     Along the same lines, it is possible to derive a similar identity for the plane wave~$u^{\rm in}(x)=  \e^{i \lambda \omega \cdot x}$
     \begin{align*}
           u^{\rm in}(x)= -\int_{|y|=R} \left( \frac{ \partial G_\lambda}{\partial r}(x,y) \,u^{\rm in}(y) - G_\lambda(x,y)
           \frac{ \partial u^{in}(y)}{\partial r}(y)\right)\, \dd \sigma (y),
     \end{align*}
     taking into account the fact that  $(D^2  - \lambda^2) u^{\rm in} = 0$.
     Subtracting the two identities gives the following representation formula for the scattered wave $u^{\rm scat}=u-u^{\rm in}$
     \begin{align*}
          u^{\rm scat}(x) &= - \int_{|y| \le R} G_\lambda(x,y) \,q(y)\,u (y) \, \dd y \\
          &\quad- \int_{|y|=R} \left( \frac{ \partial G_\lambda}{\partial r}(x,y) \,u^{\rm scat} (y) -
          G_\lambda(x,y) \frac{ \partial u^{\rm scat}}{\partial r}(y)\right) \,\dd\sigma (y).
     \end{align*}
     The Green function of the free Helmholtz equation is explicitly given by
     \begin{align}
           G_{\lambda}(x,y) = \frac{1}{4i}\bigg(\frac{\lambda}{2\pi}\bigg)^{\frac{d-2}{2}} |x-y|^{-\frac{d-2}{2}}
           H_{d/2-1}^{(1)}\big(\lambda |x-y|\big),
     \end{align}
     where $H_{d/2-1}^{(1)}$ is the Hankel function of first kind and order $d/2-1$ (cf.~Appendix \ref{App:Special}).  The asymptotic behavior of Hankel functions \eqref{hankel2} implies that for $|y|$ large enough
     \begin{align}
           G_{\lambda}(x,y) = \frac{1}{2i \lambda} \bigg(\frac{\lambda}{2\pi i}\bigg)^{\frac{d-1}{2}}
           \e^{i \lambda |x-y|} |x-y|^{-\frac{d-1}{2}}\bigg(1 + \O\bigg(\frac{1}{\lambda|x-y|}\bigg)\bigg).
     \end{align}
  Since we have
          $$ |x-y| = |y| - \frac{x \cdot y}{|y|}+\O\bigg(\frac{|x|^2}{|y|}\bigg) $$
   we find for $|y|$ large enough and fixed~$\lambda$:
          $$
           \e^{i \lambda |x-y|} = \e^{i \lambda \, | y| }  \e^{ - i \lambda \frac{y}{|y|} \cdot x}  \bigg(1+ \O\bigg(\frac{| x|^2 } {| y|}\bigg)\bigg) .
          $$
     We  therefore
    obtain for fixed $x,\lambda$ and large $R=|y|$
     \begin{align}
     \label{Far:GreenAsymptotic}
          G_\lambda(x,y)= \frac{1}{2i \lambda} \bigg(\frac{\lambda}{2\pi i}\bigg)^{\frac{d-1}{2}}
          |y|^{-\frac{d-1}{2}} \e^{i \lambda \, | y| }  \e^{ - i \lambda \frac{y}{|y|} \cdot x} \bigg(1+ \O\bigg(\frac{1} {R}\bigg)\bigg).
     \end{align}
     Analogously, we have
          $$ \frac{\d G_\lambda}{\d r}(x,y) = i \lambda c_d(\lambda)
               | y|^{-\frac{d-1}{2}} \e^{i \lambda | y| }  \e^{ - i \lambda \frac{y }
               {| y|} \cdot x } \bigg(1+ \O\bigg(\frac{1 } {R}\bigg)\bigg), $$
     with $c_d(\lambda)=\tfrac{1}{2i \lambda} (\lambda/2\pi i)^{\frac{d-1}{2}}$. This  leads to
     \begin{multline*}
          u^{\rm scat}(x) = - \int_{|y| \le R} G_\lambda(x,y) q(y) u (y) \, \dd y \\
          + \frac {c_d(\lambda)}{R^{\frac{d-1}{2}}} \int_{|y|=R} \left( \frac{ \partial u^{\rm scat}}{\partial r}  - i \lambda u^{\rm scat}\right)
          \e^{i \lambda R - i \lambda \frac{y } {| y|} \cdot  x} \bigg(1+ \O\bigg(\frac{1} {R}\bigg)\bigg)
          \, \dd\sigma (y).
      \end{multline*}
      Sommerfeld's radiation condition implies that the second right-hand side term tends to zero when   $R$ tends
      to infinity, so we get
      \begin{align*}
            u^{\rm scat}(x) &= - \int G_\lambda(x,y) q(y) u (y) \, \dd y\,.
      \end{align*}
      Using once again the asymptotic formula \eqref{Far:GreenAsymptotic} together with the fact that the Green function is symmetric
      we get
             $$ u^{\rm scat}(x)= -\frac{1}{2i \lambda} \bigg(\frac{\lambda}{2\pi i}\bigg)^{\frac{d-1}{2}}
                 \frac{\e^{i \lambda| x|}}{| x|^{\frac{d-1}{2}}}  \int q(y)u (y) \e^{ -i\lambda \frac{x }{| x|} \cdot y}
                 \bigg(1+ \O\bigg(\frac{|y|^2} {| x|}\bigg)\bigg) \,\dd y.$$
      To summarize, we have that
             $$ u^{\rm scat}(x) = \frac{\e^{i\lambda |x|}}{|x|^{\frac{d-1}{2}}} a_q\bigg(\frac{x}{|x|},\omega,\lambda\bigg)
                 +\O\bigg(\frac{1}{|x|^{\frac{d+1}{2}}}\bigg),$$
      with
            $$ a_q\bigg(\frac{x}{|x|},\omega,\lambda\bigg) = -\frac{1}{2i \lambda} \bigg(\frac{\lambda}{2\pi i}\bigg)^{\frac{d-1}{2}}
                \int q(y) u(y)  \e^{ - i \lambda \frac{x } {| x|} \cdot y}  \, \dd y. $$
      This proves that the scattered part of any solution of Helmholtz' equation subject to Sommerfeld's radiation condition
 takes the form announced in Theorem \ref{Far:ThmScat}.
\end{proof}

\begin{deft}
\label{Far:ScatDeft}
      We define the scattering amplitude associated with the potential~$q \in \CO(\R^d)$ by the smooth function $a_q : S^{d-1} \times S^{d-1} \times \R^*_+ \to \C$
      given~by
      \begin{align}
             a_q(\theta,\omega,\lambda) = -\frac{1}{2i \lambda} \bigg(\frac{\lambda}{2\pi i}\bigg)^{\frac{d-1}{2}}
             \int q(y) \phi_q(y,\omega,\lambda)  \e^{ - i \lambda \theta \cdot y}  \, \dd y.
      \end{align}
      We denote
      \begin{align*}
           A_{q}(\lambda) g(\theta) = \int_{S^{d-1}} a_q(\theta,\omega,\lambda) g(\omega) \, \dd \omega, \quad \theta \in S^{d-1},
      \end{align*}
      the corresponding operator with kernel $a_q(\theta,\omega,\lambda)$.
\end{deft}
\noindent It is easy to get an asymptotic expansion of
\begin{align*}
      \Phi_q(\lambda)g(x) = \int_{S^{d-1}} \big(\e^{i \lambda x \cdot \omega}
      +\phi_{q}^{\rm scat}(x,\omega,\lambda)\big)
       g(\omega) \, \dd \omega
\end{align*}
as $|x| \rightarrow +\infty$  using the stationary phase and
Theorem \ref{Far:ThmScat}:
\begin{multline}
\label{Far:Asymptotics}
      \Phi_q(\lambda)g =  \bigg(\frac{2 \pi}{\lambda |x|}\bigg)^{\frac{d-1}{2}} \Bigg(\e^{-i \lambda |x|} \e^{i(d-1)\frac{\pi}{4}}  g(-\theta)
       \\ +\e^{i \lambda |x|}\bigg(\e^{-i(d-1)\frac{\pi}{4}}  g(\theta) +\bigg(\frac{\lambda}{2 \pi }\bigg)^{\frac{d-1}{2}}
       A_q(\lambda)g(\theta)\bigg) + \O\bigg(\frac{1}{|x|}\bigg)\Bigg)
\end{multline}
with $\theta=x/|x| \in S^{d-1}$. The operator which maps the coefficient of $\e^{-i\lambda|x|}$
into the coefficient of $\e^{i\lambda|x|}$ is given by
     $$ g(-\theta) \mapsto i^{-d+1} \bigg(g(\theta) + \bigg(\frac{\lambda
     i}{2\pi}\bigg)^{\frac{d-1}{2}} A_q(\lambda)g(\theta)\bigg) .$$
This is, after renormalization and composition with the antipodal
map, the scattering matrix \cite{M}.
\begin{deft}
      The scattering matrix is the operator given by
           $$ S_q(\lambda) = \id + \bigg(\frac{\lambda i}{2\pi}\bigg)^{\frac{d-1}{2}} A_q(\lambda). $$
\end{deft}
\noindent Integration by parts allows to relate the values of the potential inside the domain with the scattering matrix $S_q(\lambda)$.
\begin{lem}
\label{Far:IntegralLemma}
     We have the following identities
     \begin{align*}
           \int (q_{1}-q_{2}) u_{1} \overline{u_{2}} \, \dd x &= -2i \lambda \bigg(\frac{2 \pi}{\lambda}\bigg)^{d-1}
           \int_{S^{d-1}} g_1 \, \overline{g_{2}} - S_{q_1}(\lambda) g_1 \, \overline{S_{q_{2}}(\lambda)g_{2}}
            \, \dd \omega,
           \\ \int (q_{1}-q_{2}) u_{1} u_{2} \, \dd x  &=  -2i \lambda \bigg(\frac{2 \pi}{\lambda i}\bigg)^{d-1}
           \int_{S^{d-1}} \check{g}_{2} \, S_{q_{1}}(\lambda)g_1
            - \check{g}_1 \, S_{q_{2}}(\lambda)g_{2} \, \dd
            \omega,
     \end{align*}
    for any pair of solutions%
     \footnote{We use the notation $\check{g}(\omega)=g(-\omega)=Pg(\omega)$ for the antipodal map.}
          $$ u_{1}=\Phi_{q_{1}}(\lambda)g_{1},  \quad u_{2}=\Phi_{q_{2}}(\lambda)g_{2}, $$
     to the Helmholtz equations \eqref{Far:HelmholtzPoten} related to the potentials $q_{1},q_{2}$.
\end{lem}
\begin{proof}
      Let $R$ be large enough so that the ball of radius $R$ contains the support of both potentials $q_1,q_2$.
      By Green's formula, we have
      \begin{align*}
           \int_{|x| \leq R} (q_{1}-q_{2}) u{_1} \overline{u_{2}} \, \dd x
           = \int_{|x|=R}  \d_ru_{1} \,  \overline{u_{2}} -u_{1} \, \overline{\d_ru_{2}} \, \dd \sigma(x)
     \end{align*}
     and using the asymptotic formula \eqref{Far:Asymptotics} on the functions $u_1=\Phi_{q_1}(\lambda)g_1$ and $ u_2=\Phi_{q_2}(\lambda)g_2$, we deduce
     \begin{multline*}
             \d_ru_{1} \,  \overline{u_{2}} -u_{1} \, \overline{\d_ru_{2}} =  \\ -2i \lambda \bigg(\frac{2 \pi}{\lambda R}\bigg)^{d-1}
              \Big(g_1(-\theta) \, \overline{g_2(-\theta)}
              - S_{q_1}(\lambda)g_1(\theta) \, \overline{S_{q_2}(\lambda)g_2(\theta)}\Big)
              + \O\bigg(\frac{1}{R^{d}}\bigg),
     \end{multline*}
     which implies
     \begin{multline*}
           \int_{|x| \leq R} (q_{1}-q_{2}) u{_1} \overline{u_{2}} \, \dd x  \\ = -2i \lambda \bigg(\frac{2 \pi}{\lambda}\bigg)^{d-1}
           \int_{S^{d-1}}  \Big(g_1 \, \overline{g_2} - S_{q_1}g_1 \, \overline{S_{q_2}g_2}\Big)
           \, \dd \theta + \O\bigg(\frac{1}{R}\bigg).
     \end{multline*}
     Letting $R$ tend to infinity provides the first identity.

     The proof of the second identity is similar, since
     \begin{align*}
             \d_ru_{1} \,  u_{2} -u_{1} \, \d_ru_{2} &= -2i \lambda \bigg(\frac{2 \pi}{\lambda R}\bigg)^{d-1}
              \Big(g_1(-\theta) \, S_{q_2}g_2(\theta) - S_{q_1}(\lambda)g_1(\theta) \, g_2(-\theta)\Big) \\
              &\quad + \O\bigg(\frac{1}{R^{d}}\bigg).
     \end{align*}
     This completes the proof of the lemma.
\end{proof}
\noindent Choosing $q_1=q_2=q$ in Lemma \ref{Far:IntegralLemma}, we obtain the following properties of the scattering matrix:
    $$ {}^tS_{q}(\lambda)= P \circ S_{q}(\lambda)\circ P, \quad  S_{q}(\lambda)^* S_{q}(\lambda) = \id $$
recalling that $Pg=\check{g}$ is the antipodal map. Here, $t$ denotes the
transpose and $*$ the transpose conjugate. Besides, using the
first relation together with Lemma \ref{Far:IntegralLemma}, we
finally get
 \begin{align}
 \label{Far:IntegralID}
        \int (q_{1}-q_{2}) u_{1} u_{2} \, \dd x &= 2i\lambda \bigg(\frac{2 \pi}{\lambda i}\bigg)^{\frac{d-1}{2}}
        \int_{S^{d-1}} \check{g}_{2} \, \big(A_{q_{1}}(\lambda)-A_{q_2}(\lambda)\big) g_1   \, \dd
        \omega,
\end{align}
for any pair of solutions
     $$ u_{1}=\Phi_{q_{1}}(\lambda)g_{1},  \quad u_{2}=\Phi_{q_{2}}(\lambda)g_{2}, $$
to the Helmholtz equations \eqref{Far:HelmholtzPoten} with
Sommerfeld's radiation condition (\ref{Far:Sommerfeld}) related to
the potentials $q_{1},q_{2}$.

\subsection{Stability estimates at high frequencies}\label{statementgenthm}
The first stability result we shall prove in this paper states
that low frequencies (i.e.~smaller than~$2 \lambda$) may be
recovered in a stable way from the scattering amplitude.
\begin{thm}
\label{genthm}
      For all $\eps,M,R>0$ and all $\alpha>d$ there exist~$C_{\eps},\lambda_0>0$ such that the following
      stability estimate holds
      true.

      \noindent Let~$q_1,q_2 \in \mathcal{C}^{\infty}_0(\R^d)$ be two potentials supported
      in the ball centered at~0 and of radius~$R$ such that~$\|q_1\|_{L^{\infty}},
      \|q_2\|_{L^{\infty}} \leq M$. Then for  all $\lambda \geq \lambda_0$,
            \begin{multline}
          \int_{|\xi| \leq (2-\eps) \lambda} \langle \xi \rangle^{-\alpha} \big|(\widehat{q_1-q_2})(\xi)\big|^2  \, \dd \xi \\
          \leq C_{\eps} \lambda^{3} \|a_{q_1}-a_{q_2}\|^2_{L^2} + C_{\eps} \lambda^{-2} \|q_1-q_2\|^2_{L^\infty}.
     \end{multline}
\end{thm}
\noindent The proof of Theorem~\ref{genthm} relies on estimates on solutions to the Helmoltz equations stated below, as well as on
the following lemma, stating  that one can relate the Fourier transform of the potential and the scattering amplitude.
\begin{lem}
\label{Far:AsymptLemma}
     For all $M,R>0$ there exist constants~$C,\lambda_0>0$ such that for all potentials $q\in \CO(\R^d)$ supported in the ball $B(0,R)$ and satisfying~$\|q\|_{L^{\infty}} \leq M,$ we have the following approximation:
     \begin{align}
          \bigg| a_q(\theta,\omega,\lambda) + \frac{1}{2i \lambda} \bigg(\frac{\lambda}{2\pi i}\bigg)^{\frac{d-1}{2}}
          \widehat{q}\big(\lambda(\theta-\omega)\big)\bigg| \leq  C \, \lambda^{\frac  {d-5}2}\, \|q\|^2_{L^{2}}
     \end{align}
     for all $\lambda \geq \lambda_0$.
     \end{lem}
Lemma \ref{Far:AsymptLemma} is based on the following classical
estimate on the free resolvent $R_0(\lambda)$. We refer to
\cite{B} for an exposition by Burq of an elementary proof due to
Zworski. This type of \textit{a priori} estimates have a long
history and go back to the work of Agmon  \cite{Ag} on weighted
estimates on the resolvent and the limiting absorption principle.
\begin{prop}
\label{Far:PVBoundProp}
       Let $\chi \in \mathcal{C}^{\infty}_0(\R^d)$, there exists a constant $C>0$ such that for all $\lambda>1$ we have
            $$ \|\chi R_0(\lambda)\chi\|_{\mathcal{L}(L^2(\R^d))} \leq C \lambda^{-1}. $$
\end{prop}
\begin{proof}[Proof of Lemma~\ref{Far:AsymptLemma}]
      We start with the expression defining the scattering amplitude (see Definition \ref{Far:ScatDeft}):
      $$
          a_q(\theta,\omega,\lambda)  = -\frac{1}{2i \lambda} \bigg(\frac{\lambda}{2\pi i}\bigg)^{\frac{d-1}{2}}
             \int q(y) \phi_q(y,\omega,\lambda)  \e^{ - i \lambda \theta \cdot y}  \, \dd y
      $$
             so by (\ref{decphiq}) we deduce that
             \begin{multline*}
          a_q(\theta,\omega,\lambda)  + \frac{1}{2i\lambda} \bigg(\frac{\lambda}{2\pi i}\bigg)^{\frac{d-1}{2}}  \widehat{q}\big(\lambda (\theta-\omega)\big) \\
          = -  \frac{1}{2i\lambda} \bigg(\frac{\lambda}{2\pi i}\bigg)^{\frac{d-1}{2}} \int q (x) \phi_q^{\rm scat}(x,\omega,\lambda) \, \e^{-i\lambda x \cdot \theta} \, \dd x.
     \end{multline*}
The resolvent identity reads (see for instance Formula (2.3) in \cite{M})
           \begin{equation} \label{resident} R_q(\lambda) = R_0(\lambda)  - R_0(\lambda) \, q \,R_q(\lambda)  \end{equation}
     which implies that if $\chi \in \mathcal{C}^{\infty}_0(\R^d)$ is a cutoff function which equals 1 on the ball $B(0,R)$, then
          $$ \chi R_q(\lambda)\chi = \chi R_0(\lambda)\chi  - \big(\chi R_0(\lambda) \chi\big) q \big(\chi R_q(\lambda) \chi\big). $$
     If we apply this identity to $-q\e^{i \lambda x \cdot \omega}$, we get in view of \eqref{resscat}
         $$ \chi \phi^{\rm scat}_q = -\big(\chi R_0(\lambda)\chi\big)\big(q\e^{i \lambda x \cdot \omega})  +
              \big(\chi R_0(\lambda) \chi\big) \big(q \chi \phi^{\rm scat}_q\big) $$
     and using the estimate of Proposition \ref{Far:PVBoundProp}, we obtain
         $$ \|\chi \phi^{\rm scat}_q\|_{L^2} \leq C \lambda^{-1} \|q \|_{L^2} + C \lambda^{-1} \|q\|_{L^{\infty}} \|\chi \phi^{\rm scat}_q \|_{L^2}.  $$
     Taking $\lambda \geq 2CM$ we deduce
         $$ \|\chi \phi^{\rm scat}_q\|_{L^2} \leq 2C \lambda^{-1} \|q \|_{L^2}.  $$
     Using Cauchy-Schwarz's inequality, we get
           $$  \bigg|\int q \, \phi_q^{\rm scat} \e^{-i\lambda x \cdot \omega} \, \dd x \bigg| \leq 2C \lambda^{-1} \|q\|^2_{L^2} $$
      and this completes the proof of the lemma.
     \end{proof}
\noindent Note that the main ingredient in the proof of Lemma
\ref{Far:AsymptLemma} consists in combining the estimate for the
non-perturbed resolvent (Proposition~\ref{Far:PVBoundProp}) with
the resolvent identity~(\ref{resident}), along with the fact
that~$q$ is compactly supported.

\medskip
\noindent This lemma is the basis of the reconstruction of the potential
from the scattering amplitude at high frequencies, since one can
choose $\theta, \omega \in S^{d-1}$ such that
$\lambda(\theta-\omega)=\xi$ for any fixed frequency $\xi$ with
$|\xi| \leq 2\lambda$ and let $\lambda$ tend to infinity. We use a
similar approach to prove a stability estimate at high
frequencies.
\begin{proof}[Proof of Theorem~\ref{genthm}]
      Again,  we start with the expression defining the scattering amplitude and then express the potential in terms of the
      scattering amplitude: for~$j \in \{1,2\}$ we write
      \begin{align*}
             \widehat{q_j}\big(\lambda (\theta-\omega)\big)=-2i (2\pi i)^{\frac{d-1}{2}} \lambda^{-\frac{d-3}{2}}
             a_{q_j}(\theta,\omega,\lambda) - \int q_j (x) \phi_{q_j}^{\rm scat}(x,\omega,\lambda) \, \e^{-i\lambda x \cdot \theta} \, \dd x.
      \end{align*}
      Taking the difference of the two expressions yields
      \begin{multline} \label{fourierform}
             \big|(\widehat{q_1-q_2})\big(\lambda (\theta-\omega)\big)\big| \leq 2 (2\pi )^{\frac{d-1}{2}} \lambda^{-\frac{d-3}{2}}
             \big|(a_{q_1}-a_{q_2})(\theta,\omega,\lambda)\big| \\
             + \|q_1-q_2\|_{L^2} \, \|\phi_{q_1}^{\rm scat}\|_{L^2(|x|\leq R)} +\|q_2\|_{L^2}
             \, \|\phi_{q_1}^{\rm scat}-\phi_{q_2}^{\rm scat}\|_{L^2(|x| \leq R)}
      \end{multline}
      where we recall that the supports of $q_1,q_2$ are contained in the ball of center~$0$ and radius $R>0$.
      As in the proof of Lemma \ref{Far:AsymptLemma} we have
      \begin{align}
      \label{Boundphiq}
            \|\phi_{q_1}^{\rm scat}\|_{L^2(|x|\leq R)} \leq C  \lambda^{-1} \|q_1\|_{L^2}.
      \end{align}
     In light of \eqref{resscat}, the difference of the two scattered waves satisfies
      \begin{align*}
            \phi_{q_1}^{\rm scat}-\phi_{q_2}^{\rm scat} =  - \big(R_{q_1}(\lambda)-R_{q_2}(\lambda)\big)(q_1\e^{i\lambda x \cdot \omega}\big)
          - R_{q_2}(\lambda) \big((q_1-q_2)\e^{i\lambda x \cdot \omega}\big)
      \end{align*}
      and by the resolvent identity \eqref{resident}, the first term on the right-hand side reads
           $$   R_{q_2}(\lambda)(q_1-q_2)R_{q_1}(\lambda)(q_1\e^{i\lambda x \cdot \omega}\big) = -
                  R_{q_2}(\lambda)(q_1-q_2)\,\phi^{\rm scat}_{q_1}. $$
      Let $\chi \in \mathcal{C}^{\infty}_0(\R^d)$ be a cutoff function which equals $1$ on the ball $B(0,R)$, then
      \begin{multline*}
            \chi (\phi_{q_1}^{\rm scat}-\phi_{q_2}^{\rm scat}) =  - \big(\chi R_{q_2}(\lambda)\chi\big)\big( (q_1-q_2)\phi^{\rm scat}_{q_1}\big)
            \\ - \big( \chi R_{q_2}(\lambda) \chi\big) \big((q_1-q_2)\e^{i\lambda x \cdot \omega}\big).
      \end{multline*}
      Therefore  along the same lines as in  the proof of Lemma~\ref{Far:AsymptLemma},  the resolvent identity \eqref{resident} and  Proposition \ref{Far:PVBoundProp} give  rise to
      \begin{multline} \label{useffor}
            \|\phi_{q_1}^{\rm scat}-\phi_{q_2}^{\rm scat}\|_{L^2(|x|\leq R)} \lesssim  \lambda^{-1} \|q_1-q_2\|_{L^{\infty}}
                  \|\phi_{q_1}^{\rm scat}\|_{L^2(|x|\leq R)} \\ +  \lambda^{-1} \|q_1-q_2\|_{L^{2}}.
      \end{multline}
          Indeed invoking the resolvent identity  \eqref{resident}, we infer that
          $$
           \chi R_{q_2}(\lambda)\chi  =  \chi R_{0}(\lambda)\chi   -  \big(\chi R_{0}(\lambda)\chi\big)q_2 \big(\chi R_{q_2}(\lambda)\chi\big).
          $$
    Thus applying this identity to $ (q_1-q_2)\phi^{\rm scat}_{q_1}$, we obtain
      \begin{multline*}
       \|\big(\chi R_{q_2}(\lambda)\chi\big)\big( (q_1-q_2)\phi^{\rm scat}_{q_1}\big)\|_{L^2} \leq C \lambda^{-1} \|q_1-q_2\|_{L^{\infty}}
                  \|\phi_{q_1}^{\rm scat}\|_{L^2(|x|\leq R)} \\
                  +C \lambda^{-1} \|q_2\|_{L^{\infty}} \big \|\big(\chi R_{q_2}(\lambda)\chi\big)\big( (q_1-q_2)\phi^{\rm scat}_{q_1}\big) \big\|_{L^2}, \end{multline*}
                  which implies for $\lambda > 2 C M$
                  $$ \|\big(\chi R_{q_2}(\lambda)\chi\big)\big( (q_1-q_2)\phi^{\rm scat}_{q_1}\big)\|_{L^2} \leq 2 C \lambda^{-1} \|q_1-q_2\|_{L^{\infty}}
                  \|\phi_{q_1}^{\rm scat}\|_{L^2(|x|\leq R)}.$$
We have also
                  $$ \|\big(\chi R_{q_2}(\lambda)\chi\big)\big((q_1-q_2)\e^{i\lambda x \cdot \omega}\big)\|_{L^2} \lesssim  \lambda^{-1} \|q_1-q_2\|_{L^{2}},$$
which achieves the proof of  Estimate~\eqref{useffor}. \\

\noindent      Taking into account those bounds, we obtain in view of  \eqref{fourierform} the estimate
      \begin{multline}
      \label{Far:FirstEst}
             \big|(\widehat{q_1-q_2})\big(\lambda (\theta-\omega)\big)\big| \lesssim \lambda^{-\frac{d-3}{2}}
             \big|(a_{q_1}-a_{q_2})(\theta,\omega,\lambda)\big| \\ +  \lambda^{-1} \|q_1-q_2\|_{L^{\infty}} .
      \end{multline}
      We denote $r(\xi)=\sqrt{1-|\xi|^2}$ when $|\xi| \leq 1$.  For $\xi \in \R^d$ with $|\xi| \leq 2\lambda$, choose~$\eta \in \xi^{\perp}$
      with norm $|\eta|=r(\xi/2\lambda)$. The vectors
      \begin{align*}
            \theta = \eta + \frac{\xi}{2\lambda} \virgp \quad \omega = \eta -\frac{\xi}{2\lambda} \virgp
      \end{align*}
      have length one, taking the square of \eqref{Far:FirstEst} and integrating the resulting estimate with respect to $\eta$ yields
      $$
      \begin{aligned}
&     \big|(\widehat{q_1-q_2})(\xi)\big|^2       \lesssim\lambda^{-2} \, \|q_1-q_2\|^2_{L^{\infty}} \\
  &+    {}    \lambda^{-d+3}      r(\xi/2\lambda)^{-d+2}
           \int\nolimits_{\substack{|\eta|=r(\xi/2\lambda) \\ \eta \perp \xi}}
           \Big |\big(a_{q_1}-a_{q_2}\big)\Big(\eta + \frac{\xi}{2\lambda},\eta -\frac{\xi}{2\lambda},\lambda \Big) \Big |^2\, \dd \eta  .
                 \end{aligned}
      $$
      When $|\xi| \leq (2-\eps)\lambda$ we have $r(\xi/2\lambda) \geq \sqrt{\eps(4-\eps)}/2$ therefore multiplying
      the previous inequality by $\langle \xi \rangle^{-\alpha} \leq 1$  and integrating with respect to $\xi$ we get
      $$
      \begin{aligned}
   &        \int_{|\xi| \leq (2-\eps)\lambda} \langle \xi \rangle^{-\alpha} \big|(\widehat{q_1-q_2})(\xi)\big|^2  \, \dd \xi
         \leq  C_{\eps} \lambda^{-2} \|q_1-q_2\|^2_{L^{\infty}}\\
         & \qquad    +{}C_{\eps} \lambda^{3}  \int_{|\xi| \leq 1} \int\nolimits_{\substack{|\eta|=r(\xi) \\ \eta \perp \xi}}
            \big|(a_{q_1}-a_{q_2})\big(\eta + \xi,\eta -\xi,\lambda\big)\big|^2\, \dd \eta \, \dd \xi .
      \end{aligned}
      $$
      We consider the following $2d-2$ dimensional submanifold of $S^{2d-1}$
          $$ \Sigma = \big\{(\xi,\eta) \in S^{2d-1} : \langle \xi, \eta\rangle=0 \big\} $$
      and the following diffeomorphism
      \begin{align*}
            \phi : \quad \Sigma  &\to S^{d-1}\times S^{d-1} \\
            (\xi,\eta) & \mapsto (\xi+\eta,\eta-\xi), \\
      \intertext{with inverse}
           \phi^{-1} : S^{d-1}\times S^{d-1} &\to \Sigma \\
            (\theta,\omega) &\mapsto \bigg(\frac{\theta-\omega}{2},\frac{\theta+\omega}{2}\bigg)
      \end{align*}
      for which we have
      \begin{align*}
            \int_{\Sigma} F(\eta+\xi,\eta-\xi) \, \dd \eta \wedge \dd \xi &=
            \int_{S^{d-1}\times S^{d-1}} F(\theta,\omega) \, {\phi^{-1}}^*(\dd \eta \wedge \dd \xi) \\
            &= 2^{-d} \int_{S^{d-1}\times S^{d-1}} F(\theta,\omega) \, \dd \omega \wedge \dd \theta.
      \end{align*}
      Then we finally get
 \begin{align*}
     \int_{|\xi| \leq (2-\eps)\lambda}   \langle \xi \rangle^{-\alpha} & \big|(\widehat{q_1-q_2})(\xi)\big|^2  \, \dd \xi \\
      & \leq C_{\eps}\lambda^{3}
           \iint_{S^{d-1} \times S^{d-1}} \big|a_{q_1}\big(\theta,\omega,\lambda\big)-a_{q_2}\big(\theta,\omega,\lambda\big)\big|^2
            \, \dd \theta \, \dd \omega  \\ &\quad + C_{\eps} \, \lambda^{-2} \|q_1-q_2\|^2_{L^{\infty}}
 \end{align*}
      and this completes the proof of our estimate.
\end{proof}
\noindent In particular, we recover the uniqueness of the potential from the
scattering amplitude at high frequencies.
\begin{clry}
     If  for all $(\theta,\omega,\lambda)
     $ belonging to~$ S^{d-1} \times S^{d-1} \times \R^*_+$, we have~$a_{q_1}(\theta,\omega,\lambda)=a_{q_2}(\theta,\omega,\lambda)$
     then $q_1=q_2$.
\end{clry}
\end{section}
%
%
\begin{section}{Near field pattern} \label{sect3}

\label{General case}

\subsection{Definitions and notations}

Instead of considering the Helmholtz equation on the whole Euclidean space (with Sommerfeld's radiation condition)
we focus on the Cauchy problem with Robin boundary condition on a bounded open set $\Omega \subset \R^d$
 with smooth boundary
\begin{align}
\label{Nearfield:Dirichlet}
     \begin{cases}
          (D^2-\lambda^2+q)u=0 & \text{ in } \Omega , \\
          (\d_{\nu}-i \lambda)u|_{\d \Omega} = f \in L^2(\d\Omega).
     \end{cases}
\end{align}
This problem has a unique solution $u \in H^1(\Omega)$ for all $f
\in L^{2}(\d\Omega)$. Writing a variational formulation of
\eqref{Nearfield:Dirichlet} and using a unique continuation
argument shows the uniqueness of a solution to
\eqref{Nearfield:Dirichlet}. The existence follows from Fredholm's
alternative  \cite{Mel}.
\begin{rem}
      Other classical boundary conditions are either Dirichlet or Neumann boundary conditions. However, one has to
      make the additional assumption that $\lambda^2$ is not a Dirichlet (or Neumann) eigenvalue of $D^2+q$,
      for the Dirichlet problem corresponding to \eqref{Nearfield:Dirichlet} to have a unique solution for all
      $f \in H^{\frac{1}{2}}(\d\Omega)$. Unfortunately, this does not make sense if one wants to take the high
      frequency limit $\lambda \to \infty$. To bypass this difficulty, we study the boundary problem
      \eqref{Nearfield:Dirichlet}. Indeed, this condition is
      natural. It approximates Sommerfeld's radiation condition at
      high frequencies \cite{EM,KG}.
\end{rem}
\begin{rem}
      It follows from \cite{Mel} that the unique solution $u\in
      H^1(\Omega)$ of~\eqref{Nearfield:Dirichlet} satisfies
$$
\|D u\|_{L^{2}(\Omega)}  +  \lambda \|u \|_{L^{{2}}(\Omega)}
 \leq C \big( \| q u \|_{L^{{2}}(\Omega)} + \| f\|_{L^{{2}}(\partial \Omega)}\big)
 $$
for some constant $C$ independent of $\lambda$. Therefore,  as
$\lambda \rightarrow \infty$, we have
$$
\|D u\|_{L^{2}(\Omega)}  +  \frac{\lambda}{2} \|u
\|_{L^{{2}}(\Omega)}
 \leq {C}  \| f\|_{L^{{2}}(\partial \Omega)}.
$$
\end{rem}
\begin{deft}
     The near field pattern is the map
     \begin{align*}
           \DN_{q}(\lambda) :  L^2(\d\Omega) &\to L^2(\d\Omega) \\
           f & \mapsto u|_{\d \Omega}.
     \end{align*}
\end{deft}
The typical inverse problem on the near field pattern is whether
it uniquely determines the potential $q$. This was solved (for
smooth potentials) by Sylvester and Uhlmann \cite{SU} in dimension
$d \geq 3$ in the case of the Dirichlet-to-Neumann map.
Reconstruction methods were proposed by Nachman in~\cite{N1} and
stability issues were studied by Alessandrini \cite{A}. It was
shown by Mandache in \cite{mandache} that the logarithmic
stability result of Alessandrini in~\cite{A} is optimal.

\subsection{Stability estimates at high frequencies}\label{sectiongenthm2}
The following result is the counterpart of Theorem~\ref{genthm}
in the near-field context.
\begin{thm}\label{genthm2}
 For all $M,R>0$ and all $\alpha>d$ there exist~$C,\lambda_0>0$ such that the following stability estimate holds
      true.

      \noindent Let~$q_1,q_2 \in \mathcal{C}^{\infty}_0(\R^d)$ be two potentials supported in the ball centered at~0 and of radius~$R$
      such that~$\|q_1\|_{L^{\infty}}, \|q_2\|_{L^{\infty}} \leq M$. Then for  all $\lambda \geq \lambda_0$,
      \begin{multline}
          \int_{|\xi| \leq 2\lambda} \langle \xi \rangle^{-\alpha} \big|\widehat{q_1-q_2}(\xi)\big|^2  \, \dd \xi
          \leq C \|\DN_{q_1}-\DN_{q_2}\|^2 + C \lambda^{-2} \|q_1-q_2\|^2_{L^{\infty}}.
     \end{multline}
\end{thm}
\begin{proof}
We start with two solutions $u_1,u_2$ of the equations
     $$ (D^2-\lambda^2+q_j) u_j = 0 $$
and computing
\begin{align*}
     \int_\Omega (q_1-q_2) u_1 u_2 \, \dd x &= \int_{\Omega} \Delta u_1 u_2 \, \dd x - \int_{\Omega} u_1 \Delta u_2 \, \dd x \\
     &= \int_{\d\Omega} (\d_{\nu}-i\lambda)u_1 u_2 \, \dd \sigma -  \int_{\d\Omega} u_1 (\d_{\nu}-i\lambda)u_2 \, \dd \sigma
\end{align*}
yields the formula
     $$ \int_\Omega (q_1-q_2) u_1 u_2 \, \dd x =
          \int_{\partial\Omega}\big( u_2 \, \DN_{q_1}(\lambda)u_1 -u_1\DN_{q_2}(\lambda)u_2 \big)\,  \dd \sigma. $$
Choosing $q_1=q_2$ shows that the application $\DN_{q_j}$ is symmetric, and using this additional information, we get
the formula
     $$ \int_\Omega (q_1-q_2) u_1 u_2 \, \dd x =
          \int_{\partial\Omega} u_2 \, \big(\DN_{q_1}(\lambda)-\DN_{q_2}(\lambda)\big) u_1   \dd \sigma. $$
Let us write~$q = q_1-q_2$ and choose
\begin{align*}
    u_1 &= \phi_{q_1} = \e^{i \lambda x \cdot \theta_1} + \phi^{\rm scat}_{q_1}  \\ \text{ and } \quad
    u_2 &= \phi_{q_2} = \e^{i \lambda x \cdot \theta_2} + \phi^{\rm scat}_{q_2} ,
\end{align*}
with~$\theta_{1},\theta_{2} \in S^{d-1}$.
Recall from the proof of Lemma~\ref{Far:AsymptLemma} that the scattered waves satisfy the estimate
     $$ \|\phi_{q_1}^{\rm scat}\|_{L^2(B(0,R))} +\|\phi_{q_2}^{\rm scat}\|_{L^2(B(0,R))} \leq C \lambda^{-1} .$$
It follows that
\begin{align*}
      \int_\Omega  \e^{i \lambda x \cdot  (\theta_1 + \theta_2)} q(x) \, \dd x & =
      \int_{\partial\Omega} \big(\DN_{q_1}(\lambda)-\DN_{q_2}(\lambda)\big) u_1 u_2 \, \dd\sigma \\
      &\quad -  \int_\Omega   q(x) (\phi^{\rm scat}_{q_1}+\phi^{\rm scat}_{q_2}+\phi^{\rm scat}_{q_1}\phi^{\rm scat}_{q_2}) \, \dd x
\end{align*}
and therefore
\begin{align*}
    \big|\widehat{q}\big(\lambda(\theta_{1}+\theta_{2})\big)\big| &\leq \|\DN_{q_1}(\lambda)-\DN_{q_2}(\lambda)\|
         \|u_{1}\|_{L^2(\partial \Omega)} \|u_{2}\|_{L^2(\partial \Omega)} + \frac{C}{\lambda}  \|q\|_{L^{\infty}} \\
    &\leq C \|\DN_{q_1}(\lambda)-\DN_{q_2}(\lambda)\| + \frac{C}{\lambda}  \|q\|_{L^{\infty}}
 \end{align*}
for all $(\theta_{1},\theta_{2}) \in S^{d-1} \times S^{d-1}$.

\noindent
We notice that
\begin{align*}
      S^{d-1} \times S^{d-1} &\to B(0,2) \\ (\theta_{1},\theta_{2}) &\mapsto \theta_{1}+\theta_{2}
\end{align*}
is a submersion
when $\theta_1,\theta_2$ are not colinear. This implies that
\begin{align}
\label{PointwiseEst}
     \big|\widehat{q}(\xi)\big| \leq C \|\DN_{q_1}(\lambda)-\DN_{q_2}(\lambda)\|
     + \frac{C}{\lambda}  \|q\|_{L^{\infty}}, \quad \xi \in B(0,2\lambda).
\end{align}
Multiplying this estimate by $\langle \xi \rangle^{-\alpha}/2$, taking the square and integrating on the ball $B(0,2\lambda)$
completes the proof of the theorem.
\end{proof}
\end{section}
%
%
%
\begin{section}{The case of a potential located near the boundary}\label{nearboundary}

\subsection{Definitions and notations}\label{intronearboundary}
In this section we   show, in the model case of the unit disk, that if the potential  is supported close
 to the boundary of the disk, then a larger range of frequencies may be recovered by the near
 field~$\DN_q(\lambda)$ than in the general case treated in the previous two sections.
More precisely, introducing  radial coordinates~$(x_1,x_2) = (r\cos \theta,
 r \sin \theta)$, with~$(r,\theta) \in \R^+ \times [0,2\pi[$, we
 consider the following model problem in two space dimensions:
\begin{equation}
\label{unitball}
      \begin{cases}
    \displaystyle       (D^2 - \lambda^2 +q_\lambda)u_n=0 \quad \mbox{in} \quad
            B = \big\{x \in \R^2, \: |x| \leq 1 \big\},\\
        \displaystyle    \big(\partial_r- i   \lambda\big) u_{n|\partial B} = \e^{in\theta}.
      \end{cases}
\end{equation}
We suppose  that~$q_\lambda$ is  a smooth, radial function, with   support included in~$D_\lambda^\kappa$ for some fixed constant~$\kappa>0$, where
\begin{equation}
\label{suppqlambda}
D_\lambda^\kappa= \big\{r \in [0,1], \: 1- \kappa \lambda^{-1} < r   < 1\big\} .
\end{equation}

\smallskip
\noindent
In the following for simplicity we shall drop the index~$\lambda$ in the notation of the potential.

\smallskip
  \noindent The following result shows that  in the monotone case, one can   improve on the frequency band recovered in the general case (see  Theorems~\ref{genthm} and~\ref{genthm2}). We recall that~$\|\DN_{q}(\lambda)\|$ denotes the operator norm of~$\DN_{q}(\lambda)$ in~${\mathcal L}(L^2)$.
\begin{thm}
\label{disk}
Let~$q_1$ and~$q_2$ be two smooth radial potentials supported on~$D_\lambda^\kappa$ as defined in~{\rm(\ref{suppqlambda})} and such that~$q_1 \geq q_2$.  There are positive constants~$\lambda_0 $ and~$C$ such that the following holds.

\noindent
  Let~$\lambda \mapsto K(\lambda)$ be any function such that~$ \lambda \leq  K(\lambda) $. Then  the following stability estimate is valid for all~$\lambda \geq \lambda_0$:
$$
\begin{aligned}
\int_{    |\xi| \leq K( \lambda ) } \big|(\widehat{q_1-q_2})(\xi)\big|^2  \, \dd \xi &  \leq C K^2(\lambda)\Big( \lambda^4 \|\DN_{q_1}(\lambda)-\DN_{q_2}(\lambda)\|^2  \\
&  \qquad \qquad \qquad \qquad  {} +\frac  { C_{q_1,q_2}^2} {\lambda^{4 }}   \|q_1-q_2\|^2_{L^\infty}   \Big) \,,
\end{aligned}
$$
where~$C_{q_1,q_2} = \max  \big( \|   q_1\|_{L^\infty}, \|   q_2\|_{L^\infty} ,\|   q_1\|_{L^\infty} \|   q_2\|_{L^\infty} \big).$
\end{thm}
\begin{rem}
One has trivially that
$$
\begin{aligned}
\int_{   |\xi| \leq K( \lambda ) } \big|(\widehat{q_1-q_2})(\xi)\big|^2  \, \dd \xi & \leq C  K^2( \lambda )  \|\widehat{q_1-q_2}\|^2_{L^\infty} \\
& \leq \frac{C  K^2( \lambda ) } {\lambda^2 }  \|q_1-q_2\|^2_{L^\infty}  ,
\end{aligned}
$$
so the estimate provided in Theorem~\ref{disk} is of a different nature.
\end{rem}

\begin{rem} The proof of Theorem~\ref{disk} is presented in Section~\ref{proofdisk}, as an
immediate consequence of   Lemma~\ref{lemLc_k} proved in Section~\ref{preliminariesdisk}.
That lemma relates the Laplace transform of a function to the near field operator.
It  holds in much more generality than Theorem~\ref{disk}, without
the additional assumption that~$q_1-q_2 \geq 0$. However  we are unable
to relate the Fourier transform of a function to its Laplace transform in
 general (this fact is well-known to be difficult and in general very unstable);
 for nonnegative functions however   the relation is very easy and enables us to conclude.
\end{rem}

\subsection{Proof of Theorem \ref{disk}}\label{proofdisk}
We start by stating a lemma, proved in  Section~\ref{preliminariesdisk}, which relates the Laplace transform of
a function to the near field operator. It is stated in the framework of general, non radial functions.

\smallskip
\noindent We define  the operator~$T:g \mapsto Tg$    by
$$
Tg(r) = g(1-r)
$$
 as well as the Laplace transform~${\mathcal L} $:
$$
{\mathcal L}(g)(t)= \int_{\R} g(s)\,\e^{- s\,t} \,  \dd s.
$$
For any function~$\theta \mapsto f (\theta)$ we call~$c_k(f)$ its Fourier transform, for~$k \in \Z$:
$$
 c_k(f) (r) = \int^{2\pi}_0 f(r,\theta)\,\e^{-i\,k\,\theta}\, \dd\theta.
$$

\begin{lem}\label{lemLc_k}
Let~$q_1$ and~$q_2$ be two functions supported on~$D_\lambda^\kappa$. There are three positive constants~$ \lambda_0 , K $ and~$C$ such that if~$t \geq 2K\lambda   >2 K\lambda_0 $, then
$$
\begin{aligned}
& \Big|  {\mathcal L}\big(c_k\big(T (r  {(q_1-q_2)}  \big) \big)(t)  \Big| \\
 & \quad   \leq C  \left(  \lambda^2 \, \|\DN_{q_1}(\lambda)-\DN_{q_2}(\lambda)\|
      +\frac { C_{q_1,q_2}} {\lambda^{2}}\| q_1-q_2 \|_{L^\infty}\right)
\end{aligned}$$
where~$C_{q_1,q_2} = \max  \big(  \|   q_1\|_{L^\infty}, \|   q_2\|_{L^\infty} ,\|   q_1\|_{L^\infty} \|   q_2\|_{L^\infty} \big).$
\end{lem}
 \medskip
  \begin{proof}[Proof of Theorem~\ref{disk}]
  Let~$q = q_1-q_2$. Since~$q$ is radial, we have
  $
   q(x) = Q(|x|)
  $
and
    $$
   \widehat{ q}(\xi) = {\mathcal Q}(|\xi|)
  $$
  with for all~$\rho >0$,
  $$
 {\mathcal Q}(\rho)
= \int_{0}^{2\pi} \int_0^\infty e^{-ir\rho \cos \theta} Q(r) \, r\dd r \dd\theta  .$$
  In particular recalling that~$Q  $ is nonnegative, we find that for all~$\xi \in \R^2$,
    $$
  \begin{aligned}
     |   \widehat{ q}(\xi)  | &\leq 2\pi  \, \int_0^\infty Q(r)   \, r  \dd r \\
         & \leq 2\pi  \, \int_0^\infty  e^{\zeta_0  (1-r) }e^{-\zeta_0 (1-r) } Q(r)   \,r   \dd r
        \end{aligned}
    $$
       for any~$\zeta_0\in \R$. Then we can apply Lemma~\ref{lemLc_k} to the particular case of a radial function, so choosing~$k=0$ and~$\zeta_0 = 3K \lambda_0$  we infer that
 \begin{eqnarray}
   \label{form}   |   \widehat{ q}(\xi)  |   & \leq & C e^{\zeta_0 \kappa/\lambda}  {\mathcal L} \big(T (r Q ) \big)(\zeta_0)  \\
    & \leq & C e^{3K \kappa}   \left(  \lambda^2 \, \|\DN_{q_1}(\lambda)-\DN_{q_2}(\lambda)\|
      +\frac { C_{q_1,q_2}} {\lambda^{2}}\| q_1-q_2 \|_{L^\infty}\right) \, .\nonumber
  \end{eqnarray}
  The  end of the proof of Theorem~\ref{disk} follows easily by taking the~$L^2$ norm in~$\xi$. \end{proof}

\end{section}
\subsection{Proof of Lemma~\ref{lemLc_k}}\label{preliminariesdisk}
The method of proof follows the ideas developed    in the proofs of Theorems~\ref{genthm} and~\ref{genthm2},
adapting the estimates to our special situation where the potentials are located near the boundary.
The heart of the matter consists to approximate the solutions to Helmholtz equation ~(\ref{unitball})
by  separable solutions in radial coordinates involving Bessel functions, which allows in light of
Debye's formula to relate the Laplace transform of the potential to its near field operator. \\

  More precisely, we shall look for   solutions to~(\ref{unitball})
under the following form, for~$n \in \Z$:
$$
\begin{aligned}
    u_n  (r,\theta) & = \frac{J_{|n|}(\lambda r)}{ \lambda (J'_{|n|}(\lambda) - i J_{|n|}(\lambda))}\,
     \e^{in\theta} + v_n \\
    & = z_n(r,\lambda)\, \e^{  in\theta} + v_n ,
    \end{aligned}
$$
where~$J_n$ is a Bessel function of the first kind  (see Appendix \ref{App:Special}), solution to
$$
\left( \partial_r^2 + \frac 1r \partial_r  - \frac{n^2}{r^2} +1 \right)J_n(r) = 0.
$$
Suppose that   for~$\ell \in \{1,2\}$,~$u_n^\ell$   solves~(\ref{unitball})
with potential~$  q_\ell$. Then  writing
$$
u_n^\ell (r,\theta)= z_n(r,\lambda)\, \e^{in\theta}  + v_n^\ell(r,\theta),
$$
we find that
$$
\begin{cases}
(D^2  - \lambda^2 +   q_\ell)v_n^\ell=  -     q_\ell \,  z_{n} \, \e^{  in\theta}
\quad \mbox{in}\quad B,\\
(\partial_r- i   \lambda) v^\ell_{n|\partial B} =0 .
\end{cases}
$$
 Due to Property   \eqref{hankel5}  we have if~$1-\lambda^{ -1}\leq  r\leq 1$  and~$|n| \geq K \lambda$ for~$K$ large, that~$ \big| \lambda \, z_n(r,\lambda)\big|$ is bounded, for~$\lambda\geq \lambda_0$, by a constant depending only on~$K$ and~$\lambda_0$.

 \smallskip
  \noindent
 For now on we shall denote by~$C(\lambda_0,K)  $ such a constant, which   may change from line to line.

 \smallskip

  \noindent
  Therefore, arguing as in the proof of Lemma \ref{Far:AsymptLemma} and taking advantage of the fact   that on the support of~$  q_\ell$, $r$   varies in an interval of size~$\lambda^{ -1}$,  we find  as soon as~$\lambda $ is large enough compared to~$\|   q_\ell\|_{L^\infty}$
  \begin{equation}
\label{pertvega}
\|v_n^\ell\|_{L^2} \leq \frac{C(\lambda_0,K)} \lambda \left \|    q_\ell(r,\theta) \,  z_n(r,\lambda)\right \|_{L^2} \leq  \frac{C(\lambda_0,K)} {\lambda^{\frac52}} \|    q_\ell \|_{L^\infty}.
\end{equation}
Now going back  to the computations of Section~\ref{General case} we  consider two positive integers~$n$ and $m$  such that~$n, m \geq K \lambda$, and we write
 $$
 \int_B (   q_1-   q_2) u_n^1  u_{-m} ^2 \, \dd x = \int_{\partial B} (\DN_{   q_1}(\lambda)-\DN_{   q_2}(\lambda) \e^{in\theta} \e^{-im\theta}\, \dd \theta.
 $$
 Therefore  decomposing
 \begin{align*}
     u_n^1 (r,\theta)&= z_n(r,\lambda)\, \e^{in\theta}  + v_n^1(r,\theta) \quad \text{ and } \\
     u_{-m}^2 (r,\theta)&= z_m(r,\lambda)\, \e^{-im\theta}  + v_{-m}^2(r,\theta),
\end{align*}
we get
 \begin{eqnarray}
 &&    \int_{0}^{2\pi} \big(\DN_{   q_1}(\lambda)-\DN_{   q_2}(\lambda)\big) \e^{in\theta} \e^{-im\theta}    \dd\theta \nonumber \\
     && \quad  =  \int_B (   q_1-   q_2)  \e^{i(n-m)\theta}  z_n(r,\lambda)\,z_m(r,\lambda) \, r \dd r \, \dd\theta \nonumber\\
     &&\quad\quad+ \int_B (   q_1-   q_2) \e^{in\theta}  z_n(r,\lambda)\, v_{-m}^2 \, r \dd r \, \dd\theta \label{Nq1Nq2} \\
     &&\quad\quad+  \int_B (   q_1-   q_2) \e^{-im\theta}  z_m(r,\lambda)\, v_n^1\, r \dd r \, \dd\theta \nonumber\\
     &&\quad\quad+   \int_B (   q_1-   q_2)  v_n^1  v_{-m}^2\, r \dd r \, \dd\theta.\nonumber
 \end{eqnarray}
 As $ \big| \lambda \, z_n(r,\lambda)\big|$ is bounded by a constant~$C(\lambda_0,K)$,  we have (recalling   that~$  q_1$ and $  q_2$ are compactly supported in an interval in~$r$ of size~$\lambda^{ -1}$)
 $$
 \left|  \int_B (  q_1-  q_2)  \e^{in\theta} \, z_n(r,\lambda) \, v_{-m}^2  \, r \,\dd r \, \dd\theta \right| \leq
 \frac {C(\lambda_0,K)}{ \lambda^{\frac32}} \|  q_1-  q_2\|_{L^\infty} \|v_{-m}^2\|_{L^2}
 $$
 and similarly
 $$
 \left|  \int_B (  q_1-  q_2)  \e^{-im\theta} \, z_m(r,\lambda) \, v_n^1 \,  r \,\dd r \, \dd\theta \right| \leq
 \frac {C(\lambda_0,K)} { \lambda^{\frac32}} \|  q_1-  q_2\|_{L^\infty} \|v_n^1\|_{L^2}.
 $$
Finally
 $$
  \left|  \int_B (  q_1-  q_2)   v_{n}^1  v_{-m}^2 \, r \dd r \, \dd\theta \right| \leq \|  q_1-  q_2\|_{L^\infty} \| v_{n}^1 \|_{L^2} \|v_{-m}^2\|_{L^2}.
   $$
  Thus by   \eqref{pertvega},  we get from~(\ref{Nq1Nq2}), for~$n,m \geq K \lambda >K \lambda_0 $,
 \begin{multline}
 \label{diskest1}
      \left|   \int_B (  q_1-  q_2)  \e^{i(n-m)\theta}  \, z_n(r,\lambda)\, z_m(r,\lambda)\, r\dd r \, \dd\theta  \right|  \\
      \leq \left|  \int_{0}^{2\pi} \big(\DN_{  q_1}(\lambda)-\DN_{  q_2}(\lambda)\big) \e^{in\theta} \, \e^{-im\theta} \, \dd\theta \right|
      + \frac {   C_{ q_1,  q_2}} {\lambda^{4}}  \|   q \|_{L^\infty} ,
\end{multline}
where~$ q =  q_1-  q_2$ and
$$
  C_{ q_1,  q_2} =  C(\lambda_0,K)\max  \big(\|   q_1\|_{L^\infty}, \|  q_2\|_{L^\infty} ,\|  q_1\|_{L^\infty} \|  q_2\|_{L^\infty}\big).
$$
This gives rise to
 \begin{multline}
 \label{diskest11}
      \left|   \int_B (  q_1-  q_2)  \e^{i(n-m)\theta}  \, z_n(r,\lambda)\, z_m(r,\lambda)\,  r\dd r \, \dd\theta  \right|  \\
      \leq \|\DN_{  q_1}(\lambda)-\DN_{  q_2}(\lambda)\|
      + \frac {  C_{ q_1,  q_2}} {\lambda^{4 }}  \|   q \|_{L^\infty}  ,
\end{multline}
which leads in light of \eqref{hankel5} to
\begin{multline*}
      \Big|  \int_B (  q_1-  q_2)  \e^{i(n-m)\theta}  \e^{-(n+m)(1-r)} \, r \dd r \, \dd\theta   \Big| \\
     \leq C(\lambda_0,K) \Big(  \lambda^2 \, \|\DN_{  q_1}(\lambda)-\DN_{  q_2}(\lambda)\| + \frac { C_{ q_1,  q_2}} {\lambda^{2}}  \|   q \|_{L^\infty}\Big),
\end{multline*}
for all~$n, m \geq K \lambda$.
In conclusion we have for any $k\in \Z$, any~$\ell \in \N$ and any~$j \geq 2 K\lambda> 2K \lambda_0$,
\begin{multline}
\label{diskest2}
      \left|  \int_B   (  q_1-  q_2)  \e^{-i\,k\theta}  \e^{-j(1-r)} \, r \dd r \, \dd\theta   \right| \\
       \leq C(\lambda_0,K)  \Big( \lambda^2 \, \|\DN_{  q_1}(\lambda)-\DN_{  q_2}(\lambda)\| + \frac { C_{ q_1,  q_2}} {\lambda^{2}}  \|   q \|_{L^\infty}\Big).
\end{multline}
The conclusion follows from \eqref{form}.
 Lemma~\ref{lemLc_k} is proved. \qed %
%

\section{Concluding remarks}
In this paper we have shown that the low-frequency component of
the potential can be determined in a stable way from the
scattering measurements and justified the resolution limit. We
have also proved that in the near-field we have in the monotone
case infinite resolution in reconstructing the potential near the
boundary. We think that the result holds in the general case.
However, its proof seems to be out of reach. In fact, even though
a sampling (or interpolation) formula for the Laplace transform
does exist \cite{sampling1,sampling2}, making norm-estimates
similar to those in Theorem~\ref{disk} is very challenging. Our
results can be extended in many directions. It would be very
interesting to study the limited-view case and show, as in
\cite{AGS2}, that we recover infinite resolution from near-field
measurements on the overlap of the source and receiver apertures.
Another challenging problem is to understand how probe interaction
can improve local resolution by converting evanescent modes of the
potential to propagating ones \cite{schotland}. These problems
will be the subject of forthcoming works.

\appendix

\section{Bessel functions}
\label{App:Special}
Bessel's equation arises when finding separable solutions to the Helmholtz equation in spherical coordinates,
and writes as follows:
\begin{equation}
\label{besseleq}
\left( \partial_z^2 + \frac 1 z \partial_z + \Big(1- \frac{n^2}{z^2}\Big)\right)u = 0, \quad n \in \Z.
\end{equation}
It is well known (see for instance \cite{Le, O} and the references therein) that one of the solutions
of Bessel's equation is the entire function $J_n (z)$ known as the Bessel function of the first kind of order $n$,
and defined for arbitrary~$z \in \C$ by the convergent series
$$ J_n (z) = \sum^{k=\infty}_{k=0} \frac{(-1)^k\, }{k!\,(n+k)!}\,(z/2)^{n+2k},$$
in the case where $n \in \N $ and by
$J_{- n} (z) = (-1)^n \, J_{n} (z)$.\\
%
%

\noindent To find a general solution of Bessel's equation \eqref{besseleq}, we need a second solution of \eqref{besseleq} which is linearly independent of $J_n (z)$. For such a solution, we usually choose $Y_n (z)$  the Bessel function of the second kind which is entire in the complex plane cut along the segment $]-\infty, 0]$ and defined for $n \in \N $ by
\begin{eqnarray*} Y_n (z) &=& \frac{2}{\pi} J_n (z) \log \frac{z}{2} - \frac{1}{\pi}\sum^{k=n-1}_{k=0} \frac{ (n-k-1)!}{k!}\,(\frac{z}{2})^{2k-n} \\ &- & \frac{1}{\pi}\sum^{k=\infty}_{k=0} \frac{(-1)^k\, (z/2)^{n+2k}}{k!\,(n+k)!}[\psi(k+1)+ \psi(k+n+1)],\end{eqnarray*}
where $\displaystyle \psi(m+1) = - \gamma + 1 + \frac 1 2 + ... + \frac 1 m,$
$\gamma$ being the Euler constant. We also define~$ Y_{- n}(z) = (-1)^n \, Y_{n} (z).$
\\

\noindent Since $J_n $ and $Y_n $ are linearly independent, the general expression for
solutions of \eqref{besseleq} is a linear combination of Bessel functions of the first and second kinds, i.e,
$$ u (z)= A\, J_n (z) + B\, Y_n (z), $$
where $A$ and $B$ are constants. \\

\noindent Another  basis of solutions to the differential equation \eqref{besseleq}
is given by the  Bessel functions of the third kind or Hankel functions,
denoted by $H^{(1)}_n$ and~$H^{(2)}_n$. These functions are defined  by the formulas
\begin{equation}
\label{hankel1}
      H^{(1)}_n(z)= J_n (z) + i Y_n (z) \quad \mbox{and} \quad   H^{(2)}_n(z) = J_n (z) - i Y_n (z),
\end{equation}
 where $z$ is any point of the complex plane cut along the segment $]-\infty, 0]$. The motivation for introducing the Hankel functions is that the linear combination of $J_n (z)$ and $Y_n (z)$ have very simple asymptotic expansions for large~$|z|$: it is thus well-known that
\begin{eqnarray}
\label{hankel2}
      H^{(1)}_n(z)&=& \Big(\frac 2 {\pi z}\Big)^{\frac 1 2} \e^{i\, (z-\frac {n \pi} 2  - \frac \pi 4 )}\left(1+ \O\Big(\frac{1}{|z|}\Big)\right) \, \, \mbox{and} \\
  \label{hankel3} H^{(2)}_n(z) &= & \Big(\frac 2 {\pi z}\Big)^{\frac 1 2} \e^{- i\, (z-\frac {n \pi} 2  - \frac \pi 4 )}\left(1+ \O\Big(\frac{1}{|z|}\Big)\right)
\end{eqnarray}
as $|z| \to \infty$.

\noindent Furthermore, we have the following Debye  formulas whose proof can be found
for instance in \cite[Chapter 9.4]{O}  and~\cite[Chapter~9]{as}:
$$ J_n(n \, \sech \alpha ) \sim \frac {\e^{-n(\alpha - \tanh \alpha)}} {(2 \pi n \tanh \alpha)^\frac 1 2} \left(1 + (\frac 1 8 \coth \alpha - \frac 5 {24} (\coth \alpha)^3)\frac 1 n +...\right)$$
and particulary
\begin{equation}
\label{hankel4} \
\begin{aligned}
J_n(n \,\ \sech \alpha ) = \frac {\e^{-n(\alpha - \tanh \alpha)}} {(2 \pi n \tanh \alpha)^\frac 1 2} \left(1 + \O\Big(\frac{1}{n}\Big)\right),
\\
 J'_n(n \,\ \sech \alpha ) = \frac {({ \tanh \alpha})^\frac12\e^{-n(\alpha - \tanh \alpha)}} {(4 \pi n)^\frac 1 2} \left(1 + \O\Big(\frac{1}{n}\Big)\right)
\end{aligned}
\end{equation}
as $n \to \infty$, where $\sech z $ denotes the hyperbolic secant of $z$ defined by
$$ \sech z = \frac 1 {\cosh z }$$    with $\cosh z$  the hyperbolic cosine.\\

\noindent Debye's formula gives rise to the following asymptotic behavior for the function introduced in Section~\ref{nearboundary}:
$$
z_n (\lambda,r) = \frac{J_{|n|}(\lambda r)}{ \lambda (J'_{|n|}(\lambda) - i J_{|n|}(\lambda))}
$$
defined for~$n \in \Z$. Let us prove that for~$|n| \gg \lambda$ and $r \sim 1$,
\begin{equation}
\label{hankel5}
\begin{aligned}
\mbox{Re} \, z_n (\lambda,r) \sim \frac C\lambda \e^{- n(1-r)} \left(1 + \O\Big(\frac{1}{n}\Big)\right), \\
 \mbox{and} \quad \mbox{Im} \, z_n (\lambda,r) \sim  \frac C\lambda \e^{- n(1-r)}\left(1 + \O\Big(\frac{1}{n}\Big)\right).
 \end{aligned}
\end{equation}
Without loss of generality we may assume that~$n \in \N$, then defining
$$
 \cosh \alpha_1 = \frac {n}{\lambda r} \,\, \mbox{and} \,\,  \cosh \alpha_2 = \frac {n}{\lambda },
 $$
it is easy to see  that under the above assumptions ($\lambda \ll n$ and $r \sim 1$), we have   necessarily $ \cosh \alpha_i \gg 1$  for $i \in \{1,2\}$, hence $ \alpha_i \gg 1$. This  implies~$ \cosh \alpha_i \sim  \e^{\alpha_i}$,~$ \sinh \alpha_i \sim  \e^{\alpha_i}$,
and~$\tanh \alpha_i \sim 1$, which gives rise to
$$
\mbox{Re} \, z_n (\lambda,r) = \frac{J_{n} (n \, \sech \alpha_1) J'_{n} (n \, \sech \alpha_2)  }{
\lambda (J'^2_{n}(n \, \sech \alpha_2) + J^2_{n}(n \, \sech \alpha_2)  )
}
$$
and
$$
 \mbox{Im} \, z_n (\lambda,r) = \frac{J_{n} (n \,\sech \alpha_1) J_{n} (n \, \sech \alpha_2)  }{
\lambda (J'^2_{n}(n \, \sech \alpha_2) + J^2_{n}(n \, \sech \alpha_2)  )\,\cdot
}
$$
Finally, taking advantage of \eqref{hankel4} we get
 $$
\mbox{Re} \, z_n (\lambda,r)
 \sim \frac { \e^{-n (\alpha_1 -  \alpha_2)}}{\lambda} \sim \frac{\e^{-n \log (\frac {1}{ r})}}{\lambda} \cdotp $$
 The computation is identical for~$ \mbox{Im} \, z_n (\lambda,r) $, so using the fact that $r$ is near to $1$, we obtain the desired conclusion.
\end{document}